\documentclass[]{interact}
\usepackage{latexsym}
\usepackage{anyfontsize} 
\usepackage{tikz} 
\usetikzlibrary{cd} 
\usepackage{color}
\usepackage{hyperref}
\usepackage{url}
\usepackage{breakurl}
\newcommand{\bburl}[1]{\textcolor{blue}{\url{#1}}}

\makeatletter
\newcommand{\monthyear}[1]{%
  \def\@monthyear{\uppercase{#1}}}
\newcommand{\volnumber}[1]{%
  \def\@volnumber{\uppercase{#1}}}
\makeatother

\theoremstyle{plain}
\numberwithin{equation}{section} 
\newtheorem{thm}{Theorem}[section] 
\newtheorem{theorem}[thm]{Theorem}
\newtheorem{lemma}[thm]{Lemma}

\newtheorem{definition}[thm]{Definition}

\newtheorem{corollary}[thm]{Corollary}
\newtheorem{remark}[thm]{Remark}

\numberwithin{table}{section} 
\numberwithin{figure}{section}

\begin{document}

\monthyear{Month Year}
\volnumber{Volume, Number}
\setcounter{page}{1}

\title{Weighted error-sum identities for periodic continued fractions and their generalizations}

\author{
\name{Kevin Calderon\textsuperscript{a} and Nikita Kalinin\textsuperscript{b,c}\thanks{Email address: nikita.kalin@technion.ac.il}}
\affil{\textsuperscript{a}Departamento de Matem\'aticas, Facultad de Ciencias, Universidad Nacional Aut\'onoma
de M\'exico, 04510 Ciudad de M\'exico, Mexico; 
\textsuperscript{b}Technion Israel Institute of Technology, Faculty of Mathematics,  Haifa, 3200003, Israel;
      	\textsuperscript{c}Guangdong Technion Israel Institute of Technology (GTIIT), 241 Daxue Road,
Shantou, Guangdong Province 515603, P.R. China}
}

\maketitle

{\bf Article type}: research 
\bigskip

\begin{abstract}For a purely \(N\)-periodic continued fraction
\[
\xi=[\overline{a_0,a_1,\dots,a_{N-1}}]
\]
with convergents \(h_n/k_n=[a_0,a_1,\dots,a_n]\), we obtain explicit
expressions for the weighted error sums
\[
f_\xi(s)=\sum_{n=-1}^{\infty}a_{n+1}|h_n-\xi k_n|^s.
\]
A key observation is that, for each residue class
\(r\in\{-1,0,\dots,N-2\}\), the subsequence of approximation errors
\[
\varepsilon_n=h_n-\xi k_n,\qquad n\equiv r\pmod N,
\]
forms a geometric progression. 

In addition, we extend our methods to generalized continued fractions with 
numerators $(b_n)$, obtaining Euler–type identities and weighted error–sum 
formulae for $\pi$ and $\ln 2$.  

\end{abstract}


\section{Introduction. }

For $\xi\in \mathbb R$, we can recursively define $a_0\in \mathbb{Z}$ and  $a_{n}\in \mathbb{Z}_{\geq 1}, n\geq 1$ using the following rule

\[
r_0=\xi,\qquad a_k=\lfloor r_k\rfloor,\qquad
r_{k+1}=\frac{1}{r_k-a_k}.
\]

For rational $\xi$ this process eventually terminates, and the following expression is called a continued fraction for $\xi$
    
$$\xi=\quad a_0+\cfrac{1}{a_1+\cfrac{1}{\ddots +\cfrac{1}{a_n}}}=\left[a_0, a_1, \ldots, a_n\right].$$

For an irrational real number $\xi$ we write  

$$\xi=a_0+\cfrac{1}{a_1+\cfrac{1}{a_2+\cfrac{1}{a_3+\ldots}}}=\left[a_0, a_1, \ldots\right].$$

The integers $a_i$ are called the coefficients of the continued fraction, and the convergents to $\xi$ i.e. the truncations 

$$\frac{h_n}{k_n}=a_0+\cfrac{1}{a_1+\cfrac{1}{\ddots +\cfrac{1}{a_n}}}=[a_0,\cdots,a_n],$$

may be computed using the recursions

$$h_{-2}=0, k_{-2}=1, h_{-1}=1,k_{-1}=0$$

$$h_n=a_nh_{n-1}+h_{n-2},$$

$$k_{n}=a_nk_{n-1}+k_{n-2}.$$
 
It is a classical fact that $\frac{h_n}{k_n}\to \xi$ as $n\to \infty$ which justifies the above equality sign between $\xi$ and the infinite continued fraction. If the coefficients of the continued fraction are such that for all $n> n_0$,  $a_{n+N}=a_n$, then we say that the continued fraction of $\xi$ is eventually periodic, denoted by

$$\xi=[a_0,\cdots, a_{n_0},\overline{a_{n_0+1},\cdots,a_{n_0+N}}].$$
This type of continued fractions is characterized by the following theorem 

\begin{theorem}[Lagrange]\label{th1}
A real number $\xi$ has an eventually periodic continued fraction if and only if $\xi$ is the root of a quadratic polynomial with rational coefficients.
\end{theorem} 

  Such $\xi$ is called a quadratic irrationality. If $a_{n+N}=a_{n}$ for all $n\geq 0$, then we will say that the continued fraction is purely periodic, that is, $\xi=[\overline{a_0,\cdots a_{N-1}}]$.  It follows from Theorem~\ref{th1} that $\xi=\frac{a+b\sqrt{c}}{d}$ for $a,b,c,d\in\mathbb Z$, so we define the conjugate of $\xi$ by $\bar{\xi}=\frac{a-b\sqrt{c}}{d}$. Using conjugates simplifies the characterization of the set of purely periodic continued fractions.

\begin{theorem}[Galois]\label{th2}
 Quadratic irrationality $\xi$ is purely periodic if and only if $\xi>1$ and $-1<\bar{\xi} <0$. 

\end{theorem}

In \cite{kalinin2019tropical,kalinin2024legendre}, Kalinin and Shkolnikov provide a generalization of Archimedes' exhaustion method (motivated by the study of sandpiles in convex domains), which, when applied to the triangle with vertices $(0,0)$, $(1,0)$, and $(0,\xi)$ yields the following identity for  an irrational number $\xi$:

\begin{equation}
\sum_{n=-1}^{\infty}a_{n+1}(h_n-\xi k_n)^2=\xi. 
\label{eq:frist}
\end{equation}

The terms $h_n-\xi k_n$ in \eqref{eq:frist} are known in the literature as the errors in the approximation of the continued fraction of $\xi$. 

The study of error terms arising from continued fractions has a long history,
motivated by the remarkable quality of Diophantine approximations provided by
convergents. Given a real number \(\xi\) with convergents \(h_n/k_n\), the error
\[
\varepsilon_n=h_n-\xi k_n
\]
measures the local quality of the approximation. In the quadratic case, the
continued fraction of \(\xi\) is eventually periodic, and this periodicity often
forces additional algebraic structure on the error sequence.

The first systematic analysis of error–sum functions appears in 
\cite{ridley2000error},
where the series
\[
   \sum_{n\ge0} |k_n \xi - h_n|
\]
was investigated for general $\xi\in\mathbb R$.
Subsequent developments focused particularly on the case when $\xi$ is a
quadratic irrational.  In this setting, the continued fraction expansion is
eventually periodic.  Explicit evaluations of error sums were obtained  
in \cite{elsner2014error}, where closed forms are 
derived for $\sqrt d$ and expressed in terms of Lucas and Pell numbers for $\sqrt 5$ and $\sqrt 2$.
Later, analytic refinements and weighted variants were developed in \cite{elsner2011error}
where generating functions
of the form
\[
   \sum_{n\ge0} \varepsilon_n x^n
\]
were studied and shown to admit algebraic structure for a quadratic irrationality $\xi$. An interesting identity 
$$\sum_{n\geq 0}|k_ne-h_n|=e\left(-1+10\sum_{n\geq 0}\frac{(-1)^n}{(n+1)!(2n^2+7n+3)}\right)$$
for $e$ were obtained in \cite{allouche2014variations}. 
Generalizations with ``split denominators'' and more flexible summation
structures were introduced in \cite{baruchel2016error},
offering a broader framework in which convergent--based expansions arise
naturally and reveal new identities.

More recent work has also explored fractal behavior and coprime--restricted
summation in error functions, most notably \cite{ahn2025m},
showing that such expressions can encode geometric and measure--theoretic
information when $\xi$ varies.
These results illustrate the diversity and continued activity of research in
this area.

In this work, we introduce and analyze the weighted error series
\[
   f_\xi(s) = \sum_{n\ge -1} a_{n+1}\,|h_n - \xi k_n|^s,\qquad s\ge1,
\]
for $\xi=[\overline{a_0,a_1,\dots,a_{N-1}}]$ purely periodic.
We show that the error sequence splits into $N$ geometric subsequences
\[
   \varepsilon_{mN+r} = \varepsilon_r\,\rho^m,
\qquad r\in\{-1,0,\dots,N-2\},
\]
a simple and nice result which we could not find in the literature. This leads to explicit closed formulae for $f_\xi(s)$ in terms of the
initial segment and the ratio $\rho$.  
In particular, for integer $s\ge1$, each subseries admits a simple rational
description in $\mathbb Q(\xi)$,
thereby providing a unified generalization of earlier identities for $s=1$.

\textbf{Structure of the paper.}
In Section~\ref{s2}, we establish the geometric decomposition of the 
error terms for purely periodic quadratic irrationals and derive explicit closed 
formulae for the weighted sums $f_\xi(s)$.  We present 
three independent proofs of this decomposition: via matrix actions on convergent 
pairs, via algebraic conjugation in $\mathbb{Q}(\xi)$, and via complete quotients of the 
continued fraction.

In Section~\ref{s3} we interpret the common ratio in terms
of algebraic units and explain when the corresponding unit is the fundamental
unit of the underlying quadratic order or field.  
Finally, in
Section~\ref{s4} we extend the argument to 
generalized Euler--type continued fractions with numerators $(b_n)$, leading to 
similar identities for error sums.  Using these error sums we obtain
formulae for $\pi$ and $\ln 2$.

\section{The main theorem and its three proofs}\label{s2}

\begin{theorem}\label{th3}
Suppose $\xi = [\overline{a_0, \dots, a_{N-1}}]$ is a purely periodic quadratic irrationality. Then the sequence of approximation errors $\varepsilon_n = h_n - \xi k_n$ for $n \ge -1$ can be decomposed into $N$ geometric subsequences. Namely, there exists a real number $\rho\in\mathbb Q(\xi)$ such that for any residue class $r \in \{-1, 0, \dots, N-2\}$ and any integer $m \ge 0$ we have

$$\varepsilon_{mN+r} = \varepsilon_r \cdot \rho^m.$$
\end{theorem}

We generalize identity \eqref{eq:frist} by defining the series for $s \ge 1$

\begin{equation}
f_{\xi}(s)=\sum_{n=-1}^{\infty}a_{n+1}|h_n-\xi k_n|^s.  
\label{eq:second}
\end{equation}

\begin{definition}

 Let $\xi$ be purely periodic with period $N$. For each residue class $k \in \{-1, 0, \dots, N-2\}$, we define the sub-series
 
 $$\beta_{k}(s) = \sum_{\substack{j \ge -1 \\ j \equiv k \pmod N}} |h_j-\xi k_j|^s.$$

\end{definition}

 Using this definition, we can rearrange \eqref{eq:second} as

 $$f_{\xi}(s)=\sum_{i=-1}^{N-2} a_{i+1} \beta_{i}(s).$$

\begin{theorem}\label{th4}

Let $\xi$ be purely periodic and let $s>0$.
For each residue class \(r\in\{-1,0,\dots,N-2\}\), the series \(\beta_r(s)\) is geometric with
common ratio $|\rho|^s$. Explicitly,
\[ \beta_r(s)=\frac{|h_r-\xi k_r|^s}{1-|\rho|^s}. \]
    
\end{theorem}

In this article, we will give multiple proofs of Theorem ~\ref{th3} using linear algebra and classical number theory. 

The three proofs below illustrate complementary viewpoints:
matrix dynamics, quadratic conjugation, and complete quotients.
Each framework reveals a bit different mechanism behind the geometric decay of errors.  
    
\subsection{First proof: matrix approach}

Put
\[
A_j:=
\begin{pmatrix}
a_j&1\\
1&0
\end{pmatrix},
\qquad a_{j+N}=a_j.
\]
Then
\[
\begin{pmatrix}
\varepsilon_n\\
\varepsilon_{n-1}
\end{pmatrix}
=
A_n
\begin{pmatrix}
\varepsilon_{n-1}\\
\varepsilon_{n-2}
\end{pmatrix}.
\]
For each residue class \(r\in\{-1,0,\dots,N-2\}\), define the corresponding
cyclic period matrix
\[
M_r:=A_{r+N}A_{r+N-1}\cdots A_{r+1}.
\]
The order here is important: the factor \(A_{r+1}\) acts first, then
\(A_{r+2}\), and so on. Hence
\[
\begin{pmatrix}
\varepsilon_{r+N}\\
\varepsilon_{r+N-1}
\end{pmatrix}
=
M_r
\begin{pmatrix}
\varepsilon_r\\
\varepsilon_{r-1}
\end{pmatrix},
\]
and, by iteration,
\[
\begin{pmatrix}
\varepsilon_{r+mN}\\
\varepsilon_{r+mN-1}
\end{pmatrix}
=
M_r^m
\begin{pmatrix}
\varepsilon_r\\
\varepsilon_{r-1}
\end{pmatrix}.
\]

The matrices \(M_r\) are cyclic conjugates of one another. Indeed, if
\[
M_{-1}=A_{N-1}\cdots A_0
\]
and
\[
P_r=A_rA_{r-1}\cdots A_0
\qquad (r\ge0),
\]
then
\[
M_r=P_rM_{-1}P_r^{-1}.
\]
Thus all \(M_r\) have the same two eigenvalues.

To compute these eigenvalues, it is more convenient to use the standard
convergent matrix
\[
R:=A_0A_1\cdots A_{N-1}
=
\begin{pmatrix}
h_{N-1}&h_{N-2}\\
k_{N-1}&k_{N-2}
\end{pmatrix}.
\]
Since the matrices \(A_j\) are symmetric, we have
\[
M_{-1}=R^T.
\]
Therefore \(M_{-1}\) and \(R\) have the same characteristic polynomial and hence
the same eigenvalues.

Because
\[
\xi=[a_0,a_1,\dots,a_{N-1},\xi],
\]
we have
\[
\xi=\frac{h_{N-1}\xi+h_{N-2}}{k_{N-1}\xi+k_{N-2}}.
\]
Equivalently,
\[
R
\begin{pmatrix}
\xi\\
1
\end{pmatrix}
=
(k_{N-1}\xi+k_{N-2})
\begin{pmatrix}
\xi\\
1
\end{pmatrix}.
\]
Thus
\[
u:=k_{N-1}\xi+k_{N-2}
\]
is one eigenvalue of \(R\), and hence also of \(M_{-1}\). Since \(\xi>1\),
we have \(u>1\). Moreover,
\[
\det R=\det M_{-1}=(-1)^N,
\]
so the other eigenvalue is
\[
\rho=\frac{(-1)^N}{u}
=
\frac{(-1)^N}{k_{N-1}\xi+k_{N-2}},
\]
and \(|\rho|<1\).

For each residue class \(r\), the matrix \(M_r\) is conjugate to \(M_{-1}\),
so its eigenvalues are again \(u\) and \(\rho\). On the other hand,
\[
\begin{pmatrix}
\varepsilon_{r+mN}\\
\varepsilon_{r+mN-1}
\end{pmatrix}
\longrightarrow
\begin{pmatrix}
0\\
0
\end{pmatrix}
\qquad (m\to\infty),
\]
because the convergents tend to \(\xi\). Hence the component of
\[
\begin{pmatrix}
\varepsilon_r\\
\varepsilon_{r-1}
\end{pmatrix}
\]
in the expanding eigendirection of \(M_r\) must vanish. Therefore this vector
lies in the contracting eigenspace of \(M_r\), and
\[
M_r
\begin{pmatrix}
\varepsilon_r\\
\varepsilon_{r-1}
\end{pmatrix}
=
\rho
\begin{pmatrix}
\varepsilon_r\\
\varepsilon_{r-1}
\end{pmatrix}.
\]
Consequently,
\[
\varepsilon_{r+mN}=\rho^m\varepsilon_r,
\]
as required.

\subsection{Second Proof. Without Matrices.}

Recall that since $\xi$ is purely periodic with period $N$, we have the identity \cite{niven1991introduction}:
$$
\xi = \frac{h_{N-1}\xi + h_{N-2}}{k_{N-1}\xi + k_{N-2}}.
$$
This is equivalent to the quadratic equation
$$
k_{N-1}\xi^2 + (k_{N-2} - h_{N-1})\xi - h_{N-2} = 0.
$$
Using Vieta's formulas, the product and the sum of the roots $\xi$ and its conjugate $\bar{\xi}$ are:
$$
\xi \bar{\xi} = \frac{-h_{N-2}}{k_{N-1}} \, \ \xi + \bar{\xi} = \frac{h_{N-1} - k_{N-2}}{k_{N-1}} .
$$

\begin{lemma}\label{lm1}
The quantity $u=k_{N-1}\xi + k_{N-2}$ and its conjugate $\bar{u}$ satisfy $u\bar{u} = (-1)^N$. Furthermore, $|u| > 1$ and $|\bar{u}| < 1$.
\end{lemma}

\begin{proof}

First, we establish the bound for $u$. Since $\xi$ is purely periodic, we have $\xi > 1$. The period $N \ge 1$ implies $k_{N-1} \ge 1$ and $k_{N-2} \ge 0$. Then, $u = k_{N-1}\xi + k_{N-2} > 1$, that is, $|u|>1$. We take  the product $u\bar{u}$

$$u\bar{u} = (k_{N-1}\xi + k_{N-2})(k_{N-1}\bar {\xi} + k_{N-2})= k_{N-1}^2(\xi\bar{\xi}) + k_{N-1}k_{N-2}(\xi + \bar{\xi}) + k_{N-2}^2.$$

Substituting the Vieta relations, the term $k_{N-2}^2$ cancels, leaving:
$$
u\bar{u} = k_{N-1}(-h_{N-2}) + k_{N-2}(h_{N-1}) = h_{N-1}k_{N-2} - h_{N-2}k_{N-1}.
$$
By the determinant identity for convergents, this equals $(-1)^N$. Finally, since $|u| > 1$ and $|u\bar{u}| = 1$, it follows immediately that $|\bar{u}| < 1$.
\end{proof}

\textbf{Proof of Theorem ~\ref{th3}.}

The general continued fraction identity is
\[
[a_0,\dots,a_{N-1},x]
=
\frac{h_{N-1}x+h_{N-2}}{k_{N-1}x+k_{N-2}}.
\]
For \(n\ge0\), substituting \(x=h_n/k_n\) gives
\[
\frac{h_{n+N}}{k_{n+N}}
=
\frac{h_{N-1}h_n+h_{N-2}k_n}
{k_{N-1}h_n+k_{N-2}k_n}.
\]
Multiplying by the denominator and subtracting \(\xi\), we obtain
\begin{equation}
h_{n+N}-\xi k_{n+N}
=
(h_{N-1}-\xi k_{N-1})h_n
+
(h_{N-2}-\xi k_{N-2})k_n.
\label{eq:error-expand}
\end{equation}

We now determine the two coefficients. Since
\[
h_{N-1}\xi+h_{N-2}
=
\xi(k_{N-1}\xi+k_{N-2})
=
\xi u,
\]
we have
\[
h_{N-2}-\xi k_{N-2}
=
-\xi(h_{N-1}-\xi k_{N-1}).
\]
Moreover, using
\[
h_{N-1}k_{N-2}-h_{N-2}k_{N-1}=(-1)^N,
\]
we get
\[
h_{N-1}-\xi k_{N-1}
=
\frac{(-1)^N}{u}.
\]
Substituting these two identities into \eqref{eq:error-expand}, we find
\[
h_{n+N}-\xi k_{n+N}
=
\frac{(-1)^N}{u}(h_n-\xi k_n),
\qquad n\ge0.
\]

It remains to check the initial residue \(n=-1\). Since
\[
\varepsilon_{-1}=h_{-1}-\xi k_{-1}=1,
\]
and
\[
\varepsilon_{N-1}=h_{N-1}-\xi k_{N-1}
=
\frac{(-1)^N}{u},
\]
the same relation also holds for \(n=-1\). Therefore
\[
\varepsilon_{n+N}=\rho\varepsilon_n,
\qquad n\ge -1,
\]
where
\[
\rho=\frac{(-1)^N}{u}=\bar u.
\]
Iterating gives
\[
\varepsilon_{r+mN}=\rho^m\varepsilon_r,
\qquad
r\in\{-1,0,\dots,N-2\},\quad m\ge0.
\]
Since \(|\bar u|<1\) by Lemma~\ref{lm1}, this proves the theorem.

\subsection{Third Proof. Using Complete Quotients}

We define the sequence of complete quotients $\xi_n$ by the recursion:
$$
\xi_0 = \xi, \qquad
\xi_{n+1} = \frac{1}{\xi_n - a_n}, \qquad n \ge 0,
$$
where $a_n = \lfloor \xi_n \rfloor$.

Recall the fundamental identity:
\[
\xi=\frac{\xi_{n+1}h_n+h_{n-1}}{\xi_{n+1}k_n+k_{n-1}}.
\]
Using
\[
h_nk_{n-1}-h_{n-1}k_n=(-1)^{n-1},
\]
we obtain
\[
\varepsilon_n
=
h_n-\xi k_n
=
\frac{(-1)^{n-1}}{\xi_{n+1}k_n+k_{n-1}}.
\]
Put
\[
Q_n:=\xi_{n+1}k_n+k_{n-1}.
\]
We claim that
\[
Q_n=\xi_1\xi_2\cdots \xi_{n+1}.
\]
Indeed, \(Q_{-1}=1\), and for \(n\ge0\),
\[
Q_n
=
\xi_{n+1}k_n+k_{n-1}
=
\xi_{n+1}(a_nk_{n-1}+k_{n-2})+k_{n-1}.
\]
Since
\[
\xi_n=a_n+\frac1{\xi_{n+1}},
\]
we have
\[
a_n\xi_{n+1}+1=\xi_n\xi_{n+1}.
\]
Therefore
\[
Q_n
=
\xi_{n+1}(\xi_nk_{n-1}+k_{n-2})
=
\xi_{n+1}Q_{n-1}.
\]
The claim follows by induction.

Now use periodicity. Since \(\xi_{j+N}=\xi_j\), we get
\[
\frac{Q_{n+N}}{Q_n}
=
\xi_{n+2}\xi_{n+3}\cdots \xi_{n+N+1}
=
\prod_{i=0}^{N-1}\xi_i.
\]
But
\[
\prod_{i=0}^{N-1}\xi_i
=
k_{N-1}\xi+k_{N-2}
=
u.
\]
Consequently,
\[
\frac{\varepsilon_{n+N}}{\varepsilon_n}
=
(-1)^N\frac{Q_n}{Q_{n+N}}
=
\frac{(-1)^N}{u}
=
\rho.
\]
Thus
\[
\varepsilon_{r+mN}=\rho^m\varepsilon_r,
\qquad r\in\{-1,0,\dots,N-2\},\quad m\ge0.
\]
By Lemma~\ref{lm1}, \(|\rho|<1\). We record the resulting expression for the
common ratio.  

\begin{corollary}
For \(\xi=[\overline{a_0,\dots,a_{N-1}}]\) purely periodic, the common ratio in
Theorem~\ref{th3} is
\[
\rho=\frac{(-1)^N}{u},
\qquad
u=k_{N-1}\xi+k_{N-2}.
\]
Equivalently, since \(u\bar u=(-1)^N\), one has \(\rho=\bar u\).
\end{corollary}

\subsection*{Proof of Theorem \ref{th4}.}
\begin{proof}
By Theorem~\ref{th3}, for each $r\in\{-1,0,\dots,N-2\}$, the terms of the sub-series satisfy \[
|\varepsilon_{mN+r}|=|\varepsilon_r|\,|\rho|^m. \]
Therefore,
$$
\beta_r(s)
=
\sum_{m=0}^\infty
|h_{mN+r} - \xi k_{mN+r}|^s
= |\varepsilon_r|^s \sum_{m=0}^\infty |\rho|^{sm}.
$$
Since $|\rho|<1$, this is a convergent geometric series with sum $\frac{|\varepsilon_r|^s}{1-|\rho|^s}$.
\end{proof}

\begin{corollary}
Suppose \(\xi=[\overline{a_0,\dots,a_{N-1}}]\) is a purely periodic quadratic
irrationality. Then for every integer \(s\ge1\) we have
\[
f_\xi(s)\in\mathbb Q(\xi).
\]
\end{corollary}

\begin{proof}
By Theorem~\ref{th4},
\[
f_\xi(s)
=
\sum_{r=-1}^{N-2}
a_{r+1}\frac{|\varepsilon_r|^s}{1-|\rho|^s}.
\]
For integer \(s\ge1\), all quantities in this expression belong to
\(\mathbb Q(\xi)\). Indeed, \(\varepsilon_r,\rho\in\mathbb Q(\xi)\), and since
the field is real, \(|\alpha|=\pm\alpha\) for every \(\alpha\in\mathbb Q(\xi)\).
Thus \(f_\xi(s)\in\mathbb Q(\xi)\).
\end{proof}

\begin{remark}
For non-integral values of \(s\), the arithmetic nature of \(f_\xi(s)\) is more
delicate; one expects transcendental values in many cases, but we do not pursue
this question here.
\end{remark}

\section{Connection to Period Units.}\label{s3} 
The proofs of Theorem~\ref{th3}, in particular Lemma~\ref{lm1}, show that the
period multiplier
\[
u=k_{N-1}\xi+k_{N-2}
\]
has norm \((-1)^N\). Hence \(u\) is a unit in the quadratic order generated by
the continued-fraction data, and
\[
\rho=\frac{(-1)^N}{u}
\]
is its conjugate. Under the hypotheses of Theorem~\ref{th_unit}, this period
unit is the fundamental unit of the corresponding order.

\begin{theorem}[\cite{halter2013quadratic}, \cite{trifkovic2013algebraic}]
\label{th_unit}
Let \(K=\mathbb Q(\sqrt D)\) be the real quadratic field containing the
purely periodic continued fraction
\[
\xi=[\overline{a_0,\dots,a_{N-1}}]
\]
of primitive period \(N\). Let \(\xi_0,\xi_1,\dots,\xi_{N-1}\) be the complete
quotients in one period. Then
\[
u=k_{N-1}\xi+k_{N-2}
=
\prod_{i=0}^{N-1}\xi_i
\]
is the fundamental unit of the corresponding real quadratic order. In
particular,
\[
N_{K/\mathbb Q}(u)=(-1)^N.
\]

If the displayed period is not primitive but consists of \(q\) repetitions of
the primitive period, then
\[
u=\epsilon^q,
\]
where \(\epsilon\) is the fundamental unit attached to the primitive period.
\end{theorem}

The above theorem is relevant for our purposes because we can interpret the
geometric sequence in Theorem~\ref{th3} as the action of this unit. Namely,
\[
|h_{r+\ell N}-\xi k_{r+\ell N}|
=
|(h_r-\xi k_r)u^{-\ell}|.
\]

Therefore we can express \eqref{eq:second} as
\[
f_{\xi}(s)
=
\frac{u^s}{u^s-1}
\sum_{r=-1}^{N-2}a_{r+1}|h_r-\xi k_r|^s,
\qquad s\ge1.
\]

Thus the value of the series depends only on the first \(N\) error terms and on
the period unit \(u\). When \(u\) is the fundamental unit, this gives an expression
in terms of the fundamental unit of the corresponding quadratic order. Here is a
table with examples of quadratic irrationals and their associated period units.

\begin{center}
\renewcommand{\arraystretch}{1.5}
\begin{tabular}{|c|c|c|c|c|}
\hline
$\xi$ & $D$ & $\epsilon_D$ & $u=k_{N-1}\xi+k_{N-2}$ & $N(u)$ \\
\hline
$[\overline{2}]$ 
& $2$ 
& $1+\sqrt{2}$ 
& $1+\sqrt{2}$ 
& $-1$ \\
\hline
$[\overline{1, 2}]$ 
& $3$ 
& $2+\sqrt{3}$ 
& $2+\sqrt{3}$ 
& $1$ \\
\hline
$[\overline{1}]$ 
& $5$ 
& $\dfrac{1+\sqrt{5}}{2}$ 
& $\dfrac{1+\sqrt{5}}{2}$ 
& $-1$ \\
\hline
$[\overline{2, 4}]$ 
& $6$ 
& $5+2\sqrt{6}$ 
& $5+2\sqrt{6}$ 
& $1$ \\
\hline
$[\overline{1, 1, 1, 4}]$ 
& $7$ 
& $8+3\sqrt{7}$ 
& $8+3\sqrt{7}$ 
& $1$ \\
\hline
\end{tabular}
\end{center}

Write
\[
\xi=c_0+c_1\sqrt{\mathcal D},
\qquad c_0,c_1\in\mathbb Q.
\]
Then
\[
u=k_{N-1}\xi+k_{N-2}
=
x+y\sqrt{\mathcal D},
\]
where
\[
x=k_{N-1}c_0+k_{N-2},
\qquad
y=k_{N-1}c_1.
\]
Therefore
\[
N(u)=x^2-\mathcal D y^2=(-1)^N.
\]
In general \(x\) and \(y\) need not be integers; they are the coordinates of
\(u\) in the basis \(1,\sqrt{\mathcal D}\). Thus the powers of \(u\) give
solutions of the corresponding norm equation in the relevant quadratic order.
Explicitly, if
\[
u^n=x_n+y_n\sqrt{\mathcal D},
\]
then
\[
x_n=\frac{u^n+\bar u^n}{2},
\qquad
y_n=\frac{u^n-\bar u^n}{2\sqrt{\mathcal D}}.
\]
After rewriting these coordinates in an integral basis of the order, one obtains
the usual Pell-type integral equations.

\section{Some identities for Euler continued fractions}\label{s4}

As mentioned in the introduction, the identities
\[
  \sum_{n=-1}^{\infty} a_{n+1}\,(h_n - \xi k_n)^2 = \xi,
  \qquad
  \sum_{n=-1}^{\infty} a_{n+1}\,|h_n - \xi k_n| = \xi + 1,
\]
for irrational $\xi$, derived in~\cite{kalinin2019tropical,kalinin2024legendre}, were the original
motivation for this work. In this section we consider generalized continued
fractions of the form
\[
  \xi = a_0 + \cfrac{b_1}{a_1 + \cfrac{b_2}{a_2 + \cfrac{b_3}{\ddots}}}
\]
with \(a_0\in\mathbb Z\), \(a_n\in\mathbb Z_{>0}\) for \(n\ge1\), and
\(b_n\in\mathbb Z_{>0}\), and show how analogous identities persist
in this setting.

The numerators and denominators of the $n$-th convergent to $\xi$ are defined by
$$h_{-2}=0,\ h_{-1}=1,\  h_0=a_0, \ h_{n+1}=a_{n+1}h_n+b_{n+1} h_{n-1},$$
$$k_{-2}=1,\ k_{-1}=0,\  k_0=1, \ k_{n+1}=a_{n+1}k_n+b_{n+1} k_{n-1}.$$

We define $B_n=\prod_{j=1}^n b_j$ for $n\geq 1$, with $B_{0}=1$.

\begin{theorem}\label{th6}
Let \(\xi\) be represented by the generalized continued fraction above, and put
\[
\varepsilon_n=h_n-\xi k_n.
\]
Assume that
\[
\frac{\varepsilon_M\varepsilon_{M+1}}{B_{M+1}}\to0
\qquad (M\to\infty).
\]
Then
\[
\sum_{n=-1}^{\infty}
a_{n+1}\frac{1}{B_{n+1}}\varepsilon_n^2
=
\xi.
\]

When \(b_n=1\) for all \(n\), this recovers the quadratic identity mentioned at
the beginning of this section.
\end{theorem}

\begin{proof}
Put
\[
c_n=\frac1{B_n}.
\]
For \(n\ge0\), the recurrence relations for \(h_n\) and \(k_n\) give
\[
\varepsilon_{n+1}
=
a_{n+1}\varepsilon_n+b_{n+1}\varepsilon_{n-1}.
\]
Multiplying by \(c_{n+1}\varepsilon_n\), we obtain
\[
a_{n+1}c_{n+1}\varepsilon_n^2
=
c_{n+1}\varepsilon_n\varepsilon_{n+1}
-
b_{n+1}c_{n+1}\varepsilon_n\varepsilon_{n-1}.
\]
Since
\[
b_{n+1}c_{n+1}=c_n,
\]
this becomes
\[
a_{n+1}c_{n+1}\varepsilon_n^2
=
c_{n+1}\varepsilon_n\varepsilon_{n+1}
-
c_n\varepsilon_{n-1}\varepsilon_n.
\]
Summing from \(n=0\) to \(M\), we get
\[
\sum_{n=0}^{M}a_{n+1}c_{n+1}\varepsilon_n^2
=
c_{M+1}\varepsilon_M\varepsilon_{M+1}
-
c_0\varepsilon_{-1}\varepsilon_0.
\]
Now
\[
\varepsilon_{-1}=1,
\qquad
\varepsilon_0=a_0-\xi,
\qquad
c_0=1.
\]
Assuming
\[
c_{M+1}\varepsilon_M\varepsilon_{M+1}\to0,
\]
we obtain
\[
\sum_{n=0}^{\infty}a_{n+1}c_{n+1}\varepsilon_n^2
=
\xi-a_0.
\]
Adding the missing \(n=-1\) term gives
\[
a_0c_0\varepsilon_{-1}^2=a_0.
\]
Therefore
\[
\sum_{n=-1}^{\infty}a_{n+1}c_{n+1}\varepsilon_n^2=\xi.
\]
Since \(c_{n+1}=1/B_{n+1}\), this is the desired identity.
\end{proof}

\begin{remark}
Note that if all \(b_n=1\), then \(B_{n+1}=1\) and we recover the original identity. Furthermore, the proof of Theorem~\ref{th6} relies purely on algebraic telescoping. Thus the same identity holds in any setting, real, complex, or functional, in which the continued fraction converges and
\[
c_{N+1}\varepsilon_N\varepsilon_{N+1}\to0
\qquad (N\to\infty).
\]
\end{remark}

The formula in Theorem~\ref{th6} allows us to obtain several identities from
specific generalized continued fractions. Furthermore, we can think of $\xi$ as a function represented as a continued fraction, so these identities can also be extended to functions. We will present some of these identities based on some representations as continued fractions known in the literature. While preparing the final version of this article we discovered that an equivalent fromulation of Theorem~\ref{th6} appeared in \cite{bhatn} as Eq~10.10.

\subsection{An identity involving $\pi$ }

\begin{remark}

From \cite{euler1785transformatione}

$$\arctan x = \cfrac{x}{1 + \cfrac{x^2}{3-x^2 + \cfrac{(3x)^2}{5-3x^2 + \ddots}}}$$

Setting $x=1$ and multiplying by $4$ yields the generalized continued fraction for $\pi$

$$\pi = \cfrac{4}{1 + \cfrac{1^2}{2 + \cfrac{3^2}{2 + \ddots}}}$$

\end{remark}

The coefficients are $a_0=0, a_1=1$, and $a_{n+1}=2$ for $n\ge 1$. The numerators are $b_1=4$ and $b_n=(2n-3)^2$ for $n\ge 2$. Consequently, for $n \ge 1$, we have

$$B_{n+1} = \prod_{j=1}^{n+1} b_j = 4 \cdot \prod_{j=2}^{n+1} (2j-3)^2 = 4[(2n-1)!!]^2.$$

\begin{lemma}

For $n \geq 1$, the denominators satisfy $k_n = (2n-1)!!$ and the convergents are given by the partial sums of the Leibniz series, $h_n/k_n = 4L_n:=4\sum_{j=0}^{n-1} \frac{(-1)^j}{2j+1}$.

\end{lemma}

\begin{proof}
The recurrence
\[
k_n=2k_{n-1}+(2n-3)^2k_{n-2},
\qquad n\ge2,
\]
with initial values \(k_1=1\) and \(k_2=3\), gives
\[
k_n=(2n-1)!!
\]
by induction. Similarly, putting
\[
L_n=\sum_{j=0}^{n-1}\frac{(-1)^j}{2j+1},
\]
one checks from the same recurrence and the initial value \(h_1/k_1=4\) that
\[
h_n=4k_nL_n.
\]
Hence
\[
\frac{h_n}{k_n}=4L_n.
\]
\end{proof}
The hypothesis of Theorem~\ref{th6} is satisfied in this case. Indeed,
\[
B_{n+1}=4[(2n-1)!!]^2,\qquad k_n=(2n-1)!!,
\]
and the Leibniz tail satisfies
\[
4L_n-\pi=O\left(\frac1n\right).
\]
Therefore
\[
\frac{\varepsilon_n\varepsilon_{n+1}}{B_{n+1}}
=
\frac{k_nk_{n+1}(4L_n-\pi)(4L_{n+1}-\pi)}
{4[(2n-1)!!]^2}
=
O\left(\frac1n\right)\to0.
\]

Applying Theorem~\ref{th6}, we get
\[
\sum_{n=1}^{\infty}
\frac{2}{4[(2n-1)!!]^2}
(h_n-\pi k_n)^2
+
\frac{\pi^2}{4}
=
\pi.
\]
Rearranging terms and substituting
\[
h_n-\pi k_n=k_n(4L_n-\pi),
\]
we find
\[
\pi-\frac{\pi^2}{4}
=
\frac12\sum_{n=1}^{\infty}
\left(\frac{h_n}{k_n}-\pi\right)^2
=
8\sum_{n=1}^{\infty}
\left(
\sum_{j=n}^{\infty}\frac{(-1)^j}{2j+1}
\right)^2.
\]

\begin{remark}\label{rk6}

We can express the tail $\sum_{j=n}^\infty \frac{(-1)^j}{2j+1}$ using the digamma function $\psi(z)$. This relies on the standard series representation for the difference of digamma values, found in classical references such as \cite{abramowitz1965handbook}. We start from the Weierstrass product for the Gamma function:

$$\frac{1}{\Gamma(z)} = z e^{\gamma z}\prod_{k=1}^{\infty}\left(1+\frac{z}{k}\right)e^{-z/k}.$$

Taking logarithms and differentiating yields the logarithmic derivative

$$\psi(z)=\frac{\Gamma'(z)}{\Gamma(z)}=-\gamma+\sum_{k=0}^{\infty}\left(\frac{1}{k+1}-\frac{1}{k+z}\right).$$

Subtracting $\psi(a)$ from $\psi(b)$ gives the identity valid for $a,b \notin \{0,-1,-2,\dots\}$

$$\psi(b)-\psi(a)=\sum_{k=0}^{\infty}\left(\frac{1}{k+a}-\frac{1}{k+b}\right).$$

By rearranging the alternating sum,

$$\sum_{j=n}^\infty \frac{(-1)^j}{2j+1} = \frac{(-1)^n}{4} \sum_{k=0}^\infty \left( \frac{1}{k + \frac{2n+1}{4}} - \frac{1}{k + \frac{2n+3}{4}} \right).$$

Applying the identity above with $a=\frac{2n+1}{4}$ and $b=\frac{2n+3}{4}$, we obtain

$$\sum_{j=n}^\infty \frac{(-1)^j}{2j+1} = \frac{(-1)^n}{4} \left[ \psi\left(\frac{2n+3}{4}\right) - \psi\left(\frac{2n+1}{4}\right) \right].$$

Thus, the main identity can be written in terms of special functions as

$$\pi - \frac{\pi^2}{4} = \frac{1}{2} \sum_{n=1}^\infty \left[ \psi\left(\frac{2n+3}{4}\right) - \psi\left(\frac{2n+1}{4}\right) \right]^2.$$

\end{remark}

\subsection{An Identity Involving $\ln 2$.}

Let us use Euler's continued fraction \cite{euler1785transformatione} for
\(\ln(1+x)\) at \(x=1\):
\[
\ln 2
=
\cfrac{1}{1+\cfrac{1^2}{1+\cfrac{2^2}{1+\cfrac{3^2}{1+\ddots}}}}.
\]
Here \(a_{n+1}=1\), \(b_1=1\), and
\[
b_n=(n-1)^2\qquad (n\ge2).
\]
Thus
\[
B_n=((n-1)!)^2,\qquad n\ge1,
\qquad B_0=1.
\]
A straightforward induction shows that
\[
k_n=n!,
\qquad
\frac{h_n}{k_n}=S_n,
\qquad
S_n=\sum_{j=1}^{n}\frac{(-1)^{j+1}}{j},
\]
with the convention \(S_0=0\).

The hypothesis of Theorem~\ref{th6} is satisfied. Indeed,
\[
B_{n+1}=(n!)^2,\qquad k_n=n!,
\]
and the alternating harmonic tail satisfies
\[
S_n-\ln 2=O\left(\frac1n\right).
\]
Hence
\[
\frac{\varepsilon_n\varepsilon_{n+1}}{B_{n+1}}
=
\frac{n!(n+1)!(S_n-\ln2)(S_{n+1}-\ln2)}{(n!)^2}
=
O\left(\frac1n\right)\to0.
\]
Since \(a_0=0\), Theorem~\ref{th6} gives
\[
\ln 2
=
\sum_{n=0}^\infty \frac{1}{(n!)^2}(h_n-k_n\ln2)^2
=
\sum_{n=0}^\infty (S_n-\ln2)^2.
\]

Since $S_n - \ln 2 = -\sum_{j=n+1}^\infty \frac{(-1)^{j+1}}{j}$, we derive the identity:

$$\ln 2 = \sum_{n=0}^\infty \left( \sum_{j=n+1}^\infty \frac{(-1)^{j+1}}{j} \right)^2.$$

We can derive an equivalent version in terms of the digamma function by analyzing the tail term $\sum_{j=n+1}^\infty \frac{(-1)^{j+1}}{j}$. Grouping terms in pairs allows us to rewrite the sum:

$$\sum_{j=n+1}^\infty \frac{(-1)^{j+1}}{j} =\frac{(-1)^n}{2} \sum_{k=0}^\infty \left( \frac{1}{k + \frac{n+1}{2}} - \frac{1}{k + \frac{n+2}{2}} \right).$$

Applying the difference identity derived in Remark ~\ref{rk6} with $a = \frac{n}{2} + \frac{1}{2}$ and $b = \frac{n}{2} + 1$ yields

$$\ln 2 = \frac{1}{4} \sum_{n=0}^\infty \left[ \psi\left(\frac{n}{2} + 1\right) - \psi\left(\frac{n}{2} + \frac{1}{2}\right) \right]^2.$$

\subsection{A faster identity for \texorpdfstring{\(\log 2\)}{log 2}}

We now apply Theorem~\ref{th6} to the series
\[
\log 2=\sum_{m=1}^{\infty}\frac{1}{m2^m}.
\]
Put
\[
t_m=\frac{1}{m2^m}.
\]
Euler's transformation of a series into a continued fraction gives a generalized
continued fraction whose \(n\)-th convergent is
\[
S_n=\sum_{m=1}^{n}t_m.
\]
In our notation its coefficients are
\[
a_0=0,\qquad a_1=1,\qquad b_1=t_1,
\]
\[
a_n=t_{n-1}+t_n\quad(n\ge2),
\]
and
\[
b_2=-t_2,\qquad b_n=-t_{n-2}t_n\quad(n\ge3).
\]
Equivalently,
\[
a_n=\frac{3n-1}{n(n-1)2^n}\quad(n\ge2),
\]
and
\[
b_2=-\frac18,\qquad
b_n=-\frac{1}{n(n-2)2^{2n-2}}\quad(n\ge3).
\]

Let
\[
R_n=\log2-S_n
=
\log2-\sum_{m=1}^{n}\frac{1}{m2^m}.
\]
For this Euler continued fraction one has
\[
\frac{h_n}{k_n}=S_n,
\qquad
k_n=t_1t_2\cdots t_{n-1}\quad(n\ge1),
\]
and
\[
B_{n+1}=b_1b_2\cdots b_{n+1}
=
(-1)^n k_n^2t_nt_{n+1}.
\]
Therefore
\[
\varepsilon_n=h_n-(\log2)k_n
=
-k_nR_n.
\]
Substituting this into Theorem~\ref{th6}, we get
\[
\log2
=
\frac{(\log2)^2}{t_1}
+
\sum_{n=1}^{\infty}
(-1)^n
\left(
\frac1{t_n}+\frac1{t_{n+1}}
\right)
R_n^2.
\]
Since \(t_1=1/2\) and
\[
\frac1{t_n}+\frac1{t_{n+1}}
=
n2^n+(n+1)2^{n+1}
=
(3n+2)2^n,
\]
we obtain
\[
\log2
=
2(\log2)^2
+
\sum_{n=1}^{\infty}
(-1)^n(3n+2)2^n
\left(
\log2-\sum_{m=1}^{n}\frac1{m2^m}
\right)^2.
\]

Equivalently, putting
\[
E_n=2^nR_n
=
\sum_{j=1}^{\infty}\frac{1}{(n+j)2^j},
\]
we get the cleaner form
\[
\log2
=
2(\log2)^2
+
\sum_{n=1}^{\infty}
(-1)^n\frac{3n+2}{2^n}
\left(
\sum_{j=1}^{\infty}\frac{1}{(n+j)2^j}
\right)^2.
\]

\section{Further directions.}

\subsection{Possible generalizations of the main theorem}

A natural direction is to provide multidimensional analogues for the formulas

\[
  \sum_{n=-1}^{\infty} a_{n+1}\,(h_n - \xi k_n)^2 = \xi,
  \qquad
  \sum_{n=-1}^{\infty} a_{n+1}\,|h_n - \xi k_n| = \xi + 1,
\]
replacing the classic algorithm for continued fractions with, for example, the Jacobi-Perron algorithm  \cite{hawkins2008continued, abramowitz1965handbook}.  The first proof for the pair of identities \cite{kalinin2019tropical} shows that they are geometric in nature, the first related to the area and the second to the perimeter, both of the triangle of vertices $(0,0)$, $(1,0)$, and $(0,\xi)$.  Therefore, a possible $d$-dimensional generalization would consist of a family of identities for the perimeter, area, volume, and so on, up to the $d$-volume of a $d$-simplex.  Following this, we would like to be able to give an analogue to Theorem \ref{th3}  for purely periodic irrationals of degree $d$, that is, roots of rational polynomials of degree $d$ whose expansion as a multidimensional continued fraction is purely periodic. 

This generalization is not trivial. For algorithms such as the Jacobi--Perron
algorithm, Bernstein~\cite{bernstein1964periodical,bernstein1966new} gives
examples of families of cubic irrationalities whose expansions are non-periodic,
as well as examples for which the expansion is periodic. Thus one encounters
Hermite's problem: to find a multidimensional continued-fraction algorithm for
which periodicity can be guaranteed in a way analogous to the quadratic case.
Such a result would be needed in order to obtain multidimensional analogues of
Theorems~\ref{th1} and~\ref{th2}.

A second problem is to find the correct multidimensional analogue of the error
sequence \(\varepsilon_n\). This is necessary for any genuine generalization of
the identities proved above.

As an example of what we are looking for, consider the Jacobi--Perron algorithm
for algebraic irrationalities \(\eta\) of degree \(d\) whose expansion is purely
periodic. We take
\[
\mathcal B=(1,\eta,\dots,\eta^{d-1})^T.
\]
Suppose that the \(n\)-th step of the algorithm is represented by a matrix
\(M_n\in GL_d(\mathbb Z)\), with the convention that the period matrix is
\[
M_N=M_0M_1\cdots M_{N-1}.
\]
Let \(C_n\) denote the relevant convergent column, for example the last column
in this convention. We consider the scalar quantities
\[
\vartheta_n=\langle C_n,\mathcal B\rangle
\]
as a possible multidimensional analogue of the error sequence. In this sense
one obtains the following conditional analogue of the main theorem.

\begin{theorem}\label{general}

Suppose that the JPA expansion of \(\eta\) is purely periodic and that, in the
chosen convention, the corresponding convergent columns satisfy
\[
C_{mN+r}=M_N^mC_r.
\]
Assume moreover that
\[
M_N^T\mathcal B=\mu\mathcal B
\]
for some \(\mu\) with \(|\mu|<1\). Then for any \(m\ge0\) and residue \(r\) we have

 $$\vartheta_{mN+r}=\mu^m \vartheta_{r}.$$
    
\end{theorem}

\begin{proof}
From the definition of the algorithm being purely periodic with period \(N\), the
convergent vector at step \(mN+r\) is
\[
C_{mN+r}=M_N^mC_r.
\]
Therefore
\[
\vartheta_{mN+r}
=
\langle M_N^mC_r,\mathcal B\rangle.
\]
Using
\[
\langle Ax,y\rangle=\langle x,A^Ty\rangle,
\]
we get
\[
\vartheta_{mN+r}
=
\langle C_r,(M_N^T)^m\mathcal B\rangle.
\]
By hypothesis,
\[
(M_N^T)^m\mathcal B=\mu^m\mathcal B.
\]
Hence
\[
\vartheta_{mN+r}
=
\mu^m\langle C_r,\mathcal B\rangle
=
\mu^m\vartheta_r.
\]
\end{proof}

The above theorem depends entirely on the hypothesis that \(\mathcal B\) is a
left eigenvector of the period matrix \(M_N\), which, unlike in the quadratic
case, cannot always be guaranteed.

\section*{Acknowledgements}
We thank the referees for the careful reading of our manuscript.
Kevin Calderon would like to thank the International Center for Mathematical Sciences of 
Bulgarian Academy of Sciences, and the Center for Research and Advanced Studies IPN, where this work was carried out.  He also thanks Nikita Kalinin for his tremendous support and guidance, Ernesto Lupercio, Mikhail Shkolnikov, Higinio Serrano, and Ramiro Huh-Sah for their valuable comments and suggestions.

He also wants to thank God for the wonderful mathematics he has been given.

\section*{Disclosure statement}
No conflict of interest has been reported by the authors.


\medskip
\noindent MSC2020: Primary 11A55; Secondary 11J70, 11J68, 11R11, 11B39, 33B15.


\end{document}